\documentclass[a4paper,12pt,reqno]{amsart}
\usepackage{amsmath,amsthm,amssymb}
\usepackage{graphicx, color}
\usepackage[numbers,sort&compress]{natbib}
\nonstopmode \numberwithin{equation}{section}

\newtheorem{theorem}{Theorem}[section]
\newtheorem{proposition}{Proposition}[section]

\newtheorem{lemma}{Lemma}[section]

\newtheorem{remark}{Remark}[section]

\allowdisplaybreaks

\allowdisplaybreaks

\begin{document}

\title[The $\mathtt{k}-$ Struve function]{Certain inequalities involving the $\mathtt{k}$-Struve function}

\author{Kottakkaran Sooppy  Nisar}
\address{Department of Mathematics, College of Arts and Science-Wadi Al dawser, 11991,\\
Prince Sattam bin Abdulaziz University, Saudi Arabia}
\email{ksnisar1@gmail.com, n.sooppy@psau.edu.sa}

\author{Saiful Rahman Mondal}
\address{Department of Mathematics\\
King Faisal University, Al Ahsa 31982, Saudi Arabia}
\email{smondal@kfu.edu.sa}

\author{Junesang Choi}
\address{Department of Mathematics\\
Dongguk University, Gyeongju
38066, Republic of Korea}
\email{junesang@mail.dongguk.ac.kr}


\begin{abstract}
We aim to introduce a $\mathtt{k}$-Struve function and investigate its various properties,
including mainly certain inequalities associated this function.
One of the inequalities given here is pointed out to be related to the so-called classical
Tur\'an type inequality.
We also present a differential equation, several recurrence relations, and
integral representations for this  $\mathtt{k}$-Struve function.

\end{abstract}

\maketitle
\section{Introduction and preliminaries}\label{Intro}

  D\'iaz and Pariguan \cite{Diaz} introduced and investigated the so-called
  $\mathtt{k}$-gamma function
  \begin{equation}\label{k-gamma}
    \Gamma_{\mathtt{k}}(x):= \int_0^\infty t^{x-1}e^{-\frac{t^{\mathtt{k}}}{\mathtt{k}}}\, dt
    \quad \left(\Re(x)>0;\,\mathtt{k} \in \mathbb{R}^+ \right).
  \end{equation}
  Here and in the following, let $\mathbb{C}$,   $\mathbb{R}$,  $\mathbb{R}^+$, $\mathbb{N}$, and $\mathbb{Z}^{-}$ be the sets of complex numbers, real numbers,  positive real numbers, positive integers,  and negative integers, respectively, and let $\mathbb{N}_0:=\mathbb{N} \cup \{0\}$.
For various properties of the $\mathtt{k}$-gamma function and its applications to generalize other related functions
 such as $\mathtt{k}$-beta function and $\mathtt{k}$-digamma function, we refer the interested reader, for example, to
 \cite{Diaz, Kwara14, Mubeen13} and the references cited therein.

 Nantomah and Prempeh \cite{Kwara14} defined the $k$-digamma function $\Psi_\mathtt{k}:=\Gamma_\mathtt{k}'/\Gamma_\mathtt{k}$
whose series representation is given as follows:
\begin{equation}\label{def-digamma}
 \Psi_\mathtt{k}(t):=\frac{\log \mathtt{k}-\gamma}{\mathtt{k}}-\frac{1}{t}
+\sum_{n=1}^\infty \frac{t}{n\mathtt{k}(n\mathtt{k}+t)} \quad \left(\mathtt{k} \in \mathbb{R}^+;\, t
\in \mathbb{C} \setminus \mathtt{k} \mathbb{Z}^{-} \right),
\end{equation}
where $\gamma$ is the Euler-Mascheroni constant (see, e.g., \cite[Section 1.2]{Sr-Ch-12}).
A calculation yields
\begin{equation}\label{def-digamma-2}
\Psi_\mathtt{k}'(t)=\sum_{n=0}^\infty  \frac{1}{(n\mathtt{k}+t)^2} \quad \left(\mathtt{k},\, t \in \mathbb{R}^+ \right).
\end{equation}
Clearly, $\Psi_\mathtt{k}(t)$ is increasing on $(0, \infty)$.

Tur\'an \cite{Turan} proved that the Legendre polynomials $P_n(x)$ satisfy the following determinantal inequality
\begin{equation}\label{turan-ine}
\left|
          \begin{array}{cc}
            P_n(x) & P_{n+1}(x) \\
            P_{n+1}(x) & P_{n+2}(x) \\
          \end{array}
      \right|
\leq  0 \quad \left(-1 \le x \le 1;\, n \in \mathbb{N}_0\right),
\end{equation}
where the equality occurs only when $x = \pm 1$.
 Recently, many researchers have applied the above classical inequality  \eqref{turan-ine}  in various polynomials and functions such as ultraspherical polynomials, Laguerre polynomials,  Hermite polynomials,
Bessel functions of the first kind, modified Bessel functions, and polygamma functions.
Karlin and Szeg\"o \cite{Karlin} named such determinants as in \eqref{turan-ine}   Tur\'anians.

In this paper, we  consider the following  $\mathtt{k}$-Struve  function (cf. \cite[p.496, Entry 12.1.3]{ABRAMOWITZ})
\begin{equation}\label{k-Struve}
\mathtt{S}_{\nu,c}^{\mathtt{k}}(x):=\sum_{r=0}^{\infty}\frac{(-c)^r}{\Gamma_{\mathtt{k}}(r\mathtt{k}+\nu+\frac{3\mathtt{k}}{2})\Gamma(r+\frac{3}{2})}
\left(\frac{x}{2}\right)^{2r+\frac{\nu}{\mathtt{k}}+1}.
\end{equation}
Then we investigate the $\mathtt{k}$-Struve  function \eqref{k-Struve} as follows:
 We establish certain inequalities involving $\mathtt{S}_{\nu,c}^{\mathtt{k}}$,
one of which is shown to be related to the Tur\'an type inequality;
 we show that the $\mathtt{k}$-Struve  function
satisfies a second order non-homogeneous differential equation; we present an integral representation and recurrence relations
for the $\mathtt{k}$-Struve  function.

\section{Inequalities}\label{sec2}

The modified $\mathtt{k}$-Struve function is given as
\begin{equation}\label{MSF}
 L_{\nu}^{\mathtt{k}}(x):= \mathtt{S}_{\nu,-1}^{\mathtt{k}}(x),
\end{equation}
which is normalized and denoted by $\mathcal{L}_{\nu}^{\mathtt{k}}$ as follows:
\begin{equation}\label{modified-kstruve}
\mathcal{L}_{\nu}^{\mathtt{k}}(x)=\left(\frac{2}{x}\right)^{\frac{\nu}{\mathtt{k}}}
\Gamma_{\mathtt{k}}\left(\nu+\frac{3\mathtt{k}}{2}\right)L_{\nu}^{\mathtt{k}}(x)
=\sum_{r=0}^{\infty}f_{r}(\nu,\mathtt{k} )\, x^{2r+1},
\end{equation}
where
\begin{align*}
f_{r}(\nu,\mathtt{k} ): =\frac{\Gamma_{\mathtt{k}} \left(\nu+\frac{3\mathtt{k}}{2}\right)}
{\Gamma_{\mathtt{k}}(r\mathtt{k}+\nu+\frac{3}{2})\Gamma(r+\frac{3}{2})\,2^{2r+1}}.
\end{align*}

Here, we investigate monotonicity and log-convexity involving the $\mathcal{L}_{\nu}^{\mathtt{k}}$.
To do this, we recall some known useful properties
which are given in the following lemma (see \cite{Biernacki-Krzy}).

\vskip 3mm

\begin{lemma}\label{lemma:1}
Consider the power series $f(x)=\sum_{\mathtt{k}=0}^\infty a_\mathtt{k} x^\mathtt{k}$ and $g(x)=\sum_{\mathtt{k}=0}^\infty b_\mathtt{k} x^\mathtt{k}$, where $a_\mathtt{k} \in \mathbb{R}$ and $b_\mathtt{k} \in \mathbb{R}^+ $ $(\mathtt{k} \in \mathbb{N}_0)$. Further suppose that both series converge on $|x|<r$. If the sequence $\{a_\mathtt{k}/b_\mathtt{k}\}_{\mathtt{k}\geq 0}$ is increasing $($or decreasing$)$, then the function $x \mapsto f(x)/g(x)$ is also increasing $($or decreasing$)$ on $(0,r)$.

If both $f$ and  $g$ are even, or both are odd functions, then the above results will be applicable.
\end{lemma}

\vskip 3mm

\begin{theorem}\label{Th4.1}
Let $\mathtt{k} \in \mathbb{R}^+$ be fixed. Then the following statements hold.
\begin{enumerate}
\item[(i)]For $\nu \geq \mu>-{3 \mathtt{k}}/{2}$, then the function $x \mapsto \mathcal{{L}}_{\mu}^{\mathtt{k}}(x)/\mathcal{L}_{\nu}^{\mathtt{k}}(x)$ is increasing on $\mathbb{R}$. \label{a}

\item[(ii)] The function $\nu \mapsto \mathcal{L}_{\nu}^{\mathtt{k}}(x)$ is decreasing for a fixed $x \in [0, \infty)$
   and increasing for a fixed $x \in (-\infty,0)$.  Also, the function $\nu \mapsto \mathcal{L}_{\nu}^{\mathtt{k}}(x)$ is
 log-convex on $(-{3 \mathtt{k}}/{2}, \infty)$ for a fixed $x \in \mathbb{R}^+$. \label{c}

 \item[(iii)]  The function $\nu \mapsto L_{\nu+\mathtt{k}}^{\mathtt{k}}(x)/L_{\nu}^{\mathtt{k}}(x)$ is decreasing on $(-{3 \mathtt{k}}/{2}, \infty)$ for a fixed $x \in \mathbb{R}^+$. \label{b}
\end{enumerate}
\end{theorem}

\begin{proof} To prove (i),
  recall the series in \eqref{modified-kstruve}. Clearly,
\begin{align*}
\frac{\mathcal{L}_{\nu}^{\mathtt{k}}(x)}{\mathcal{L}_{\mu}^{\mathtt{k}}(x)} =\frac{\sum\limits_{r=0}^{\infty}f_{r}(\nu, \mathtt{k})\, x^{2r+1}}{\sum\limits_{r=0}^{\infty}f_{r}(\mu,\mathtt{k})\,x^{2r+1}}.
\end{align*}
Denote $w_{r}:={f_{r}(\nu, \mathtt{k})}/{f_{r}(\mu, \mathtt{k})}$. Then
\begin{align*}
w_{r}=\frac{\Gamma_{\mathtt{k}}(\nu+\frac{3}{2}\mathtt{k})\Gamma_{\mathtt{k}}(r\mathtt{k}+\mu+\frac{3}{2}\mathtt{k})}
{\Gamma_{\mathtt{k}}(r\mathtt{k}+\nu+\frac{3}{2}\mathtt{k})\Gamma_{\mathtt{k}}(\mu+\frac{3}{2}\mathtt{k})}.
\end{align*}
Appealing to the relation \eqref{FR-Gamma-k}, we find
\begin{align*}
\frac{w_{r+1}}{w_{r}}&=\frac{\Gamma_{\mathtt{k}}(\nu+\frac{3}{2}\mathtt{k})\Gamma_{\mathtt{k}}(r\mathtt{k}+\mathtt{k}+\mu+\frac{3}{2}\mathtt{k})}{\Gamma_{\mathtt{k}}(r\mathtt{k}+\mathtt{k}+\nu+\frac{3}{2}\mathtt{k})\Gamma_{\mathtt{k}}(\mu+\frac{3}{2}\mathtt{k})}\frac{\Gamma_{\mathtt{k}}(\mu+\frac{3}{2}\mathtt{k})\Gamma_{\mathtt{k}}(r\mathtt{k}+\nu+\frac{3}{2}\mathtt{k})}{\Gamma_{\mathtt{k}}
(r\mathtt{k}+\mu+\frac{3}{2}\mathtt{k})\Gamma_{\mathtt{k}}(\nu+\frac{3}{2}\mathtt{k})}\\
&=\frac{(r\mathtt{k}+\mu+\frac{3}{2}\mathtt{k})\Gamma(r\mathtt{k}+\mu+\frac{3}{2}\mathtt{k})\Gamma_{\mathtt{k}}(r\mathtt{k}+\nu+\frac{3}{2}\mathtt{k})}{(r\mathtt{k}+\nu+\frac{3}{2}\mathtt{k})\Gamma(r\mathtt{k}+\nu+\frac{3}{2}\mathtt{k})\Gamma_{\mathtt{k}}(r\mathtt{k}+\mu+\frac{3}{2}\mathtt{k})}
=\frac{r\mathtt{k}+\mu+\frac{3}{2}\mathtt{k}}{r\mathtt{k}+\nu+\frac{3}{2}\mathtt{k}}\leq 1,
\end{align*}
whose last inequality is valid from  the condition $\nu\geq \mu>-{3}\mathtt{k}/2$.
Finally, the result (i)  follows from  Lemma \ref{lemma:1}.

\vskip 3mm

For (ii), since $2\nu>-3\mathtt{k}$, we first observe the coefficients $f_{r}(\nu, \mathtt{k})>0$ for all $r \in \mathbb{N}_0$.
Then the logarithmic derivative of $f_{r}(\nu, \mathtt{k})$ with respect to $\nu$ is
\begin{equation*}
 \frac{f_{r}^{'}(\nu, \mathtt{k})}{f_{r}(\nu, \mathtt{k})}=\Psi_{\mathtt{k}}\Big(\nu+\frac{3}{2}\mathtt{k}\Big)-\Psi_{\mathtt{k}}\Big(r\mathtt{k}+\nu+\frac{3}{2}\mathtt{k}\Big)\leq 0,
\end{equation*}
whose last inequality follows from \eqref{def-digamma}.
 Since $f_{r}(\nu, \mathtt{k})>0$ $\left(r \in \mathbb{N}_0;\,2\nu>-3\mathtt{k}\right)$, ${f_{r}}'(\nu, \mathtt{k})\leq 0$ $\left(r \in \mathbb{N}_0;\,2\nu>-3\mathtt{k}\right)$.
Hence $\nu \mapsto f_{r}(\nu, \mathtt{k})$ is decreasing on $(-3 \mathtt{k}/2,\infty)$. This implies that, for $\mu\geq \nu>-{3}\mathtt{k}/{2}$,
\begin{equation*}
  \sum_{r=0}^{\infty}f_{r}(\nu, \mathtt{k})x^{2r+1}\geq \sum_{r=0}^{\infty}f_{r}(\mu, \mathtt{k})x^{2r+1} \quad (x \in [0,\infty))
\end{equation*}
and
\begin{equation*}
  \sum_{r=0}^{\infty}f_{r}(\nu, \mathtt{k})x^{2r+1}\leq \sum_{r=0}^{\infty}f_{r}(\mu, \mathtt{k})x^{2r+1} \quad (x \in (-\infty, 0)).
\end{equation*}
This proves the first statement of (ii).

\vskip 3mm
In view of \eqref{def-digamma-2}, we have
\begin{equation*}
  \frac{\partial^{2}}{\partial \nu^{2}}\left(\log(f_{r}(\nu, \mathtt{k}))\right)
   = \sum_{n=0}^{\infty}\left\{\frac{1}{(n\mathtt{k}+\nu+\frac{3}{2}\mathtt{k})^{2}}
    -\frac{1}{(n\mathtt{k}+r\mathtt{k}+\nu+\frac{3}{2}\mathtt{k})^{2}}\right\}\geq 0
\end{equation*}
for all $\mathtt{k} \in \mathbb{R}^+$ and $2\nu>-3\mathtt{k}$.
Therefore $\nu\mapsto f_{r}(\nu)$ is log-convex on $(-3\mathtt{k}/2 ,\infty)$.
Since a sum of log-convex functions is log-convex, the second statement of (ii)
is proved.

\vskip 3mm
For (iii), it is obvious from (i) that
\begin{equation}\label{(iii)-pf1}
 \frac{d}{dx}\left(\frac{\mathcal{L}_{\mu}^{\mathtt{k}}(x)}
{\mathcal{L}_{\nu}^{\mathtt{k}}(x)}\right)\geq 0,
\end{equation}
for all $x \in \mathbb{R}^+$ and $\nu\geq \mu>-{3}\mathtt{k}/2$.
In view of the relation \eqref{modified-kstruve},
\eqref{(iii)-pf1} is equivalent to
\begin{align}\label{11}
\left(x^{-\frac{\mu}{\mathtt{k}}}\mathtt{L}_{\mu}^{\mathtt{k}}(x)\right)'
\left(x^{-\frac{\nu}{\mathtt{k}}}\mathtt{L}_{\nu}^{\mathtt{k}}(x)\right)-
\left(x^{-\frac{\mu}{\mathtt{k}}}\mathtt{L}_{\mu}^{\mathtt{k}}(x)\right)
\left(x^{-\frac{\nu}{\mathtt{k}}}\mathtt{L}_{\nu}^{\mathtt{k}}(x)\right)'\geq 0
\end{align}
for all $x \in \mathbb{R}^+$ and $\nu\geq \mu>-{3}\mathtt{k}/2$.

Considering \eqref{MSF} and setting $c=-1$ in \eqref{rec-2}
gives
\begin{equation}\label{DF-pf-1}
 \frac{d}{dx}\left(x^{-\frac{\nu}{\mathtt{k}}}\mathtt{L}_{\nu}^{\mathtt{k}}(x)\right)
=\frac{{2}^{-\frac{\nu}{\mathtt{k}}}}
{\sqrt{\pi}\; \Gamma_{\mathtt{k}}
(\nu+\frac{\mathtt{3 k}}{2})}
+\,x^{-\frac{\nu}{\mathtt{k}}}\mathtt{L}_{\nu+\mathtt{k}}^{\mathtt{k}}(x).
\end{equation}
Applying \eqref{DF-pf-1} to the inequality \eqref{11} and using \eqref{modified-kstruve}, we obtain
\begin{equation} \label{last-pf-2}
  \aligned
 & x^{-\frac{\mu+\nu}{\mathtt{k}}}\left\{\mathtt{L}_{\mu+\mathtt{k}}^{\mathtt{k}}(x)
\mathtt{L}_{\nu}^{\mathtt{k}}(x)-\mathtt{L}_{\nu+\mathtt{k}}^{\mathtt{k}}(x)
\mathtt{L}_{\mu}^{\mathtt{k}}(x)\right\}\\
&\geq  \frac{{2}^{-\frac{\nu}{\mathtt{k}}}}
{\sqrt{\pi}\; \Gamma_{\mathtt{k}}
(\nu+\frac{\mathtt{3 k}}{2})} x^{-\frac{\mu}{\mathtt{k}}}\mathtt{L}_{\mu}^{\mathtt{k}}(x)
- \frac{{2}^{-\frac{\mu}{\mathtt{k}}}}
{\sqrt{\pi}\; \Gamma_{\mathtt{k}}
(\mu+\frac{\mathtt{3 k}}{2})} x^{-\frac{\nu}{\mathtt{k}}}\mathtt{L}_{\nu}^{\mathtt{k}}(x)\\
&=  \frac{{2}^{-\frac{\mu+\nu}{\mathtt{k}}}}
{\sqrt{\pi}\; \Gamma_{\mathtt{k}}
(\mu+\frac{\mathtt{3 k}}{2})\Gamma_{\mathtt{k}}
(\nu+\frac{\mathtt{3 k}}{2})}\left(\mathcal{L}_{\mu}^{\mathtt{k}}(x)
- \mathcal{L}_{\nu}^{\mathtt{k}}(x)\right) \geq 0
    \endaligned
\end{equation}
for all $x \in \mathbb{R}^+$ and $\nu\geq \mu>-{3}\mathtt{k}/2$. Here, the last inequality in \eqref{last-pf-2}
follows from the first statement of (ii). Also, we find from \eqref{last-pf-2} that
\begin{equation*}
 \frac{L_{\mu +\mathtt{k}}^{\mathtt{k}}(x)}{L_{\mu}^{\mathtt{k}}(x)} -  \frac{L_{\nu +\mathtt{k}}^{\mathtt{k}}(x)}{L_{\nu}^{\mathtt{k}}(x)}\geq 0
\end{equation*}
for all $x \in \mathbb{R}^+$ and $\nu\geq \mu>-{3}\mathtt{k}/2$.
This proves (iii).
\end{proof}

\vskip 3mm

\begin{remark}\label{thm4-1-rmk-1} \rm

One of the most significant  consequences of Theorem \ref{Th4.1}  is the Tur\'an-type inequality for the function $\mathcal{L}_{\nu}^{\mathtt{k}}$.
The log-convexity of $\mathcal{L}_{\nu}^{\mathtt{k}}$ (the last statement of (ii) in  Theorem \ref{Th4.1})
implies
\begin{equation}\label{thm4-1-rmk-1-eq1}
  \mathcal{L}_{\alpha \nu_1+(1-\alpha)\nu_2}^{\mathtt{k}}(x) \leq \left(\mathcal{L}_{\nu_1}^{\mathtt{k}}\right)^{\alpha}(x)
\left(\mathcal{L}_{\nu_2}^{\mathtt{k}}\right)^{1-\alpha}(x)
\end{equation}
\begin{equation*}
  \left(\alpha \in [0,1];\,\, x,\, \mathtt{k} \in \mathbb{R}^+ ;\,\, \nu_1,\, \nu_2 \in (-3\mathtt{k}/2, \infty) \right).
\end{equation*}
 Choosing $\alpha=1/2$ and setting $\nu_1=\nu-\mathtt{a}$ and $\nu_2=\nu+\mathtt{a}$
 for some $\mathtt{a} \in \mathbb{R}$ in \eqref{thm4-1-rmk-1-eq1}
 yields the following reversed Tur\'an type inequality (cf. \eqref{turan-ine})
 \begin{equation}\label{eqn-Turan-1}
\left(\mathcal{L}_{\nu}^{\mathtt{k}}(x)\right)^2 - \mathcal{L}_{\nu-\mathtt{a}}^{\mathtt{k}}(x)
\mathcal{L}_{\nu+\mathtt{a}}^{\mathtt{k}}(x) \leq 0
\end{equation}
 \begin{equation*}
  \left(x,\, \mathtt{k} \in \mathbb{R}^+ ;\,\, \mathtt{a} \in \mathbb{R},\, \nu \in \left(|\mathtt{a}|-
  3 \mathtt{k}/2, \infty \right) \right).
 \end{equation*}

 \end{remark}

\section{Formulae for the $\mathtt{k}$-Struve function}\label{sec3}

  Here, we present a differential equation, and recurrence relations
  regarding the $\mathtt{k}$-Struve function $\mathtt{S}_{\nu, c}^{\mathtt{k}}$  \eqref{k-Struve}.

   \vskip 3mm
\begin{proposition}\label{pro1}
Let $\mathtt{k} \in \mathbb{R}^+$ and $2\nu >-3 \mathtt{k}$. Then the $\mathtt{k}$-Struve function $\mathtt{S}_{\nu, c}^{\mathtt{k}}$ \eqref{k-Struve}
 satisfies  the following second-order non-homogeneous  differential equation
\begin{equation}\label{diff-eqn}
x^2\,\frac{d^2y}{dx^2}+ x\, \frac{dy}{dx}+\frac{1}{\mathtt{k}^2}\left(c\, \mathtt{k}\,x^2-
\nu^2\right)y= \frac{4\left(\frac{x}{2}\right)^{\frac{\nu}{\mathtt{k}}+1}}{ \mathtt{k} \Gamma_\mathtt{k}( \nu+ \frac{ \mathtt{k}}{2}) \Gamma( \frac{1}{2})}.
\end{equation}
\end{proposition}

\begin{proof}
By using the $\mathtt{k}$-Struve function $\mathtt{S}_{\nu, c}^{\mathtt{k}}$  \eqref{k-Struve}
  and the functional relation
  \begin{equation}\label{FR-Gamma-k}
    \Gamma_{\mathtt{k}} (x + \mathtt{k})=x\, \Gamma_{\mathtt{k}} (x),
  \end{equation}
we find
\begin{align*}
&x^2 \frac{d^2}{dx^2}\mathtt{S}_{\nu, c}^\mathtt{k} (x)+ x \frac{d}{dx}\mathtt{S}_{\nu, c}^\mathtt{k} (x)\\&=\sum_{r=0}^\infty \frac{(-c)^r\left(2r+\frac{\nu}{\mathtt{k}}+1\right)^2 }{ \Gamma_\mathtt{k}(r \mathtt{k}+ \nu+ \frac{3 \mathtt{k}}{2}) \Gamma(r +\frac{3}{2})} \left(\frac{x}{2}\right)^{2r+\frac{\nu}{\mathtt{k}}+1}\\
&= \sum_{r=0}^\infty \frac{(-c)^r (2r+1)\left(2r+\frac{2\nu}{\mathtt{k}}+1\right) }{ \Gamma_\mathtt{k}(r \mathtt{k}+ \nu+ \frac{3 \mathtt{k}}{2}) \Gamma(r +\frac{3}{2})} \left(\frac{x}{2}\right)^{2r+\frac{\nu}{\mathtt{k}}+1}+
\frac{\nu^2}{\mathtt{k}^2} \mathtt{S}_{\nu, c}^\mathtt{k} (x)\\
&=\frac{4}{\mathtt{k}} \sum_{r=0}^\infty \frac{(-c)^r \left(r+\frac{1}{2}\right)\left(r \mathtt{k}+{\nu}+\frac{\mathtt{k}}{2}\right) }{ \Gamma_\mathtt{k}(r \mathtt{k}+ \nu+ \frac{3 \mathtt{k}}{2}) \Gamma(r +\frac{3}{2})} \left(\frac{x}{2}\right)^{2r+\frac{\nu}{\mathtt{k}}+1}+
\frac{\nu^2}{\mathtt{k}^2}\mathtt{S}_{\nu, c}^\mathtt{k} (x)\\
&=\frac{4}{\mathtt{k}}\frac{\left(\frac{x}{2}\right)^{\frac{\nu}{\mathtt{k}}+1}}{ \Gamma_\mathtt{k}( \nu+ \frac{ \mathtt{k}}{2}) \Gamma( \frac{1}{2})} +\frac{4}{\mathtt{k}} \sum_{r=1}^\infty \frac{(-c)^r }{ \Gamma_\mathtt{k}(r \mathtt{k}+ \nu+ \frac{ \mathtt{k}}{2}) \Gamma(r +\frac{1}{2})} \left(\frac{x}{2}\right)^{2r+\frac{\nu}{\mathtt{k}}+1}+
\frac{\nu^2}{\mathtt{k}^2}\mathtt{S}_{\nu, c}^\mathtt{k} (x)\\
 &=\frac{4}{\mathtt{k}}\frac{\left(\frac{x}{2}\right)^{\frac{\nu}{\mathtt{k}}+1}}{ \Gamma_\mathtt{k}( \nu+ \frac{ \mathtt{k}}{2}) \Gamma( \frac{1}{2})}- \frac{c x^2}{\mathtt{k}}\mathtt{S}_{\nu, c}^\mathtt{k} (x) +
\frac{\nu^2}{\mathtt{k}^2}\mathtt{S}_{\nu, c}^\mathtt{k} (x).
\end{align*}
This shows that $y=\mathtt{S}_{\nu, c}^{\mathtt{k}}(x)$ the differential equation \eqref{diff-eqn}.
\end{proof}

\vskip 3mm

\begin{theorem}\label{thm1}
Let $\mathtt{k} \in \mathbb{R}^+$ and $2\nu >-3 \mathtt{k}$. Then the following recurrence relations
hold true:
\begin{align}\label{rec-1}
\frac{d}{dx}\left(x^{\frac{\nu}{\mathtt{k}}}\mathtt{S}_{\nu,c}^{\mathtt{k}}(x)\right)
&=\frac{1}{\mathtt{k}} x^{\frac{\nu}{\mathtt{k}}}\mathtt{S}_{{\nu-\mathtt{k}}}^{\mathtt{k}}(x);\\ \label{rec-2}
\frac{d}{dx}\left(x^{-\frac{\nu}{\mathtt{k}}}\mathtt{S}_{\nu,c}^{\mathtt{k}}(x)\right)
&=\frac{{2}^{-\frac{\nu}{\mathtt{k}}}}
{\sqrt{\pi}\; \Gamma_{\mathtt{k}}
(\nu+\frac{\mathtt{3 k}}{2})}
-c\,x^{-\frac{\nu}{\mathtt{k}}}\mathtt{S}_{\nu+\mathtt{k}, c}^{\mathtt{k}}(x);\\ \label{rec-3}
\frac{1}{\mathtt{k}} \mathtt{S}_{{\nu-\mathtt{k}}}^{\mathtt{k}}(x)-
c\,\mathtt{S}_{\nu+\mathtt{k}, c}^{\mathtt{k}}(x)&=
2\frac{d}{dx}\mathtt{S}_{\nu,c}^{\mathtt{k}}(x)-\frac{ {(x/2)}^{\frac{\nu}{\mathtt{k}}}}
{\sqrt{\pi}\; \Gamma_{\mathtt{k}}
(\nu+\frac{\mathtt{3 k}}{2})};\\ \label{rec-4}
\frac{1}{\mathtt{k}} \mathtt{S}_{{\nu-\mathtt{k}}}^{\mathtt{k}}(x)+
c\,\mathtt{S}_{\nu+\mathtt{k}, c}^{\mathtt{k}}(x)&=
\frac{2\nu}{x\,\mathtt{k} }\mathtt{S}_{\nu,c}^{\mathtt{k}}(x)+\frac{ {(x/2)}^{\frac{\nu}{\mathtt{k}}}}
{\sqrt{\pi}\; \Gamma_{\mathtt{k}}
(\nu+\frac{\mathtt{3 k}}{2})}.
\end{align}
\end{theorem}

\begin{proof}
From \eqref{diff-eqn} we have
\[ x^{\frac{\nu}{\mathtt{k}}}\mathtt{S}_{\nu,c}^{\mathtt{k}}(x)=
\sum_{r=0}^{\infty}\frac{(-c)^r}
{\Gamma_{\mathtt{k}}(r\mathtt{k}+\nu+\frac{3\mathtt{k}}{2})\Gamma(r+\frac{3}{2})}
\frac{x^{2r+\frac{2\nu}{\mathtt{k}}+1}}{2^{2r+\frac{\nu}{\mathtt{k}}+1}},\]
which, upon differentiating with respect to $x$ and using the relation  \eqref{FR-Gamma-k},
yields
\begin{align*}
& \frac{d}{dx}\left(x^{\frac{\nu}{\mathtt{k}}}\mathtt{S}_{\nu,c}^{\mathtt{k}}(x)\right)
=\sum_{r=0}^{\infty}\frac{(-c)^r(2r+\frac{2\nu}{\mathtt{k}}+1)}
{\Gamma_{\mathtt{k}}(r\mathtt{k}+\nu+\frac{3\mathtt{k}}{2})\Gamma(r+\frac{3}{2})}
\frac{x^{2r+\frac{2\nu}{\mathtt{k}}}}{2^{2r+\frac{\nu}{\mathtt{k}}+1}}\\
&\hskip 10mm =\frac{x^{\frac{\nu}{\mathtt{k}}}}{\mathtt{k}}\sum_{r=0}^{\infty}\frac{(-c)^r}
{\Gamma_{\mathtt{k}}
(r\mathtt{k}+\nu+\frac{\mathtt{k}}{2})\Gamma(r+\frac{3}{2})}
\left(\frac{x}{2}\right)^{2r+\frac{\nu}{\mathtt{k}}}
=\frac{1}{\mathtt{k}}\;  x^{\frac{\nu}{\mathtt{k}}}\; \mathtt{S}_{\nu-\mathtt{k} , c}^{\mathtt{k}}(x).
\end{align*}
This prove  \eqref{rec-1}.
We can establish the result \eqref{rec-2} by a similar argument as in the proof of \eqref{rec-1}.
We omit the details.

Similarly, from \eqref{k-Struve}, we obtain
\begin{equation}\label{rec-1-1}
  x\frac{d}{dx}\mathtt{S}_{\nu,c}^{\mathtt{k}}(x)+\frac{\nu}{\mathtt{k}}
\mathtt{S}_{\nu,c}^{\mathtt{k}}(x)
=\frac{x}{\mathtt{k}} \mathtt{S}_{{\nu-\mathtt{k}}}^{\mathtt{k}}(x)
\end{equation}
and
\begin{equation}\label{rec-1-2}
 x\frac{d}{dx}\mathtt{S}_{\nu,c}^{\mathtt{k}}(x)-\frac{\nu}{\mathtt{k}}
\mathtt{S}_{\nu,c}^{\mathtt{k}}(x) =\frac{x {(x/2)}^{\frac{\nu}{\mathtt{k}}}}
{\sqrt{\pi}\; \Gamma_{\mathtt{k}}
(\nu+\frac{\mathtt{3 k}}{2})}
-c\,x\mathtt{S}_{\nu+\mathtt{k}, c}^{\mathtt{k}}(x).
\end{equation}
Adding and subtracting each side of \eqref{rec-1-1} and \eqref{rec-1-2}
 yields,  respectively, the results \eqref{rec-3} and \eqref{rec-4}.

\end{proof}

\section{Integral representations}\label{sec4}

Here, we present two integral representations for the function $\mathtt{S}_{\nu, c}^{\mathtt{k}}$.
\vskip 3mm

\begin{theorem}\label{thm2} Let $\mathtt{k} \in \mathbb{R}^+$, $\Re (\nu)> -\frac{\mathtt{k}}{2}$, and $\alpha\in \mathbb{R} \setminus \{0\}$. Then
  \begin{align}\label{thm2-eq1}
\mathtt{S}_{\nu, \alpha^2}^{\mathtt{k}} (x) =\frac{2\sqrt{\mathtt{k}}}{\alpha^2\sqrt{\pi}\Gamma_{\mathtt{k}}\left(\nu+\frac{\mathtt{k}}{2}\right)} \left(\frac{x }{2}\right)^{\frac{\nu}{\mathtt{k}}}\int_0^1   (1-t^2)^{\frac{\nu}{\mathtt{k}}-\frac{1}{2}} \sin\left(\frac{\alpha x t}{ \sqrt{\mathtt{k}}}\right) dt.
\end{align}
and
\begin{align}\label{thm2-eq2}
\mathtt{S}_{\nu, -\alpha^2}^{\mathtt{k}} (x) =\frac{2\sqrt{\mathtt{k}}}{\sqrt{\pi}\Gamma_{\mathtt{k}}\left(\nu+\frac{\mathtt{k}}{2}\right)} \left(\frac{x }{2}\right)^{\frac{\nu}{\mathtt{k}}}\int_0^1   (1-t^2)^{\frac{\nu}{\mathtt{k}}-\frac{1}{2}} \sinh\left(\frac{\alpha x t}{ \sqrt{\mathtt{k}}}\right) dt.
\end{align}
In particular, we have
\begin{equation}\label{pc-thm2-eq1}
  1-\cos\left(\frac{\alpha x}{\sqrt{\mathtt{k}}}\right)= \frac{\alpha}{\mathtt{k}}\sqrt{\frac{\pi x}{2}}\mathtt{S}_{\frac{1}{2}, \alpha^2}^{\mathtt{k}} (x)
\end{equation}
and
\begin{equation}\label{pc-thm2-eq2}
 \cosh\left(\frac{\alpha x}{\sqrt{\mathtt{k}}}\right)-1= \frac{\alpha}{\mathtt{k}}\sqrt{\frac{\pi x}{2}}\mathtt{S}_{\frac{\nu}{\mathtt{k}}, -\alpha^2}^{\mathtt{k}} (x).
\end{equation}

\end{theorem}

\begin{proof}
 We begin by recalling  the $\mathtt{k}$-beta function (see \cite{Diaz})
\begin{align}\label{eqn-6}
\mathtt{B}_{\mathtt{k}}(x, y)=\frac{ \Gamma_{\mathtt{k}}(x) \Gamma_{\mathtt{k}}(y)}{\Gamma_{\mathtt{k}}(x+y)}=\frac{1}{\mathtt{k}} \int_0^1 t^{\frac{x}{\mathtt{k}}-1}(1-t)^{\frac{y}{\mathtt{k}}-1} \,dt
\end{align}
\begin{equation*}
  \left(\mathtt{k} \in \mathbb{R}^+;\, \min \{\Re(x),\,\Re(y)\} >0   \right).
\end{equation*}
Replacing $t$ by $t^2$  on the right-hand-sided integral in \eqref{eqn-6}, we obtain
\begin{align}\label{eqn-7}
\mathtt{B}_{\mathtt{k}}(x, y)=\frac{2}{\mathtt{k}} \int_0^1 t^{\frac{2x}{\mathtt{k}}-1}(1-t^2)^{\frac{y}{\mathtt{k}}-1} \, dt.
\end{align}
Setting $x= (r+1)\mathtt{k}$ and
$y= \nu+\mathtt{k}/2$ in \eqref{eqn-7} gives
\begin{align}\label{eqn-11}
\frac{ 1}{\Gamma_{\mathtt{k}}(r \mathtt{k}+\nu+\mathtt{k})}=\frac{2}{\Gamma_{\mathtt{k}}\left(\left(r+1\right) \mathtt{k}\right) \Gamma_{\mathtt{k}}\left(\nu+\frac{\mathtt{k}}{2}\right) } \int_0^1 t^{2r}(1-t^2)^{\frac{\nu}{\mathtt{k}}-\frac{1}{2}} \, dt.
\end{align}
Applying the known identity
\begin{align}\label{eqn-12}
\Gamma_{\mathtt{k}}(\mathtt{k} x)= \mathtt{k}^{x-1} \Gamma(x)
\end{align}
and the Legendre  duplication formula (see \cite{ABRAMOWITZ, Andrews-Askey, Sr-Ch-12})
\begin{equation}\label{eqn:Legendre-duplication}
  \mathrm{\Gamma}{(z)}\mathrm{\Gamma}{\left(z+\tfrac{1}{2}\right)}= 2^{1-2z}\; \sqrt{\pi}\; \mathrm{\Gamma}{(2z)}
\end{equation}
to the function  $\mathtt{S}_{\nu, c}^{\mathtt{k}}$ with \eqref{eqn-11}, we get
\begin{equation}\label{eqn-13}
  \mathtt{S}_{\nu, c}^{\mathtt{k}} (x)
  = \frac{2\sqrt{\mathtt{k}}}{\sqrt{\pi}\Gamma_{\mathtt{k}}\left(\nu+\frac{\mathtt{k}}{2}\right)} \left(\frac{x }{2}\right)^{\frac{\nu}{\mathtt{k}}}\int_0^1   (1-t^2)^{\frac{\nu}{\mathtt{k}}-\frac{1}{2}} \sum_{r=0}^\infty \frac{(-c)^r }{ (2r+1)!} \left(\frac{x t}{ \sqrt{\mathtt{k}}}\right)^{2r+1}\, dt.
\end{equation}
Finally, setting $c=\pm \alpha^2$ $(\alpha \in \mathbb{R}\setminus \{0\})$ in \eqref{eqn-13} yields, respectively,
 the desired results \eqref{thm2-eq1} and \eqref{thm2-eq2}.
Further, setting  $\nu=\mathtt{k}/2$ in \eqref{thm2-eq1} and \eqref{thm2-eq2}
yields, respectively,  the desired results \eqref{pc-thm2-eq1} and \eqref{pc-thm2-eq2}.

\end{proof}

\end{document}